\documentclass[12pt,reqno]{amsart}

\usepackage{multirow}
\usepackage{hhline}
\usepackage{float}

\usepackage{mathtools}
\usepackage{times}
\usepackage[T1]{fontenc}
\usepackage{mathrsfs}
\usepackage{latexsym}
\usepackage[dvips]{graphics}
\usepackage[titletoc, title]{appendix}
\usepackage{amsmath,amsfonts,amsthm,amssymb,amscd}
\usepackage[dvipsnames]{xcolor}
\usepackage{hyperref}
\usepackage{amsmath}
\usepackage[utf8]{inputenc}

\usepackage{color}
\usepackage{breakurl}

\usepackage{comment}
\newcommand{\bburl}[1]{\textcolor{blue}{\url{#1}}}

\usepackage{caption}

\newtheorem{thm}{Theorem}[section]

\newtheorem{cor}[thm]{Corollary}

\newtheorem{lem}[thm]{Lemma}
\newtheorem{prop}[thm]{Proposition}

\newtheorem{rek}[thm]{Remark}
\newtheorem{prob}[thm]{Problem}

\usepackage[utf8]{inputenc}

\DeclareFixedFont{\ttb}{T1}{txtt}{bx}{n}{12} 
\DeclareFixedFont{\ttm}{T1}{txtt}{m}{n}{12}  

\usepackage{color}
\definecolor{deepblue}{rgb}{0,0,0.5}
\definecolor{deepred}{rgb}{0.6,0,0}
\definecolor{deepgreen}{rgb}{0,0.5,0}

\usepackage{listings}

\newcommand\pythonstyle{\lstset{
language=Python,
basicstyle=\ttm,
morekeywords={self},              
keywordstyle=\ttb\color{deepblue},
emph={MyClass,__init__},          
emphstyle=\ttb\color{deepred},    
stringstyle=\color{deepgreen},
frame=tb,                         
showstringspaces=false
}}

\lstnewenvironment{python}[1][]
{
\pythonstyle
\lstset{#1}
}
{}

\newcommand\pythoninline[1]{{\pythonstyle\lstinline!#1!}}

\definecolor{ao}{rgb}{0.0, 0.5, 0.0}

\numberwithin{equation}{section}

\DeclareFontFamily{U}{mathx}{}
\DeclareFontShape{U}{mathx}{m}{n}{<-> mathx10}{}
\DeclareSymbolFont{mathx}{U}{mathx}{m}{n}
\DeclareMathAccent{\widehat}{0}{mathx}{"70}
\DeclareMathAccent{\widecheck}{0}{mathx}{"71}

\title{Problems Regarding a Pair of Diophantine Equations}

\author[H. V. Chu]{H\`ung Vi\d{\^e}t Chu}
\email{\textcolor{blue}{\href{mailto:hchu@wlu.edu}{hchu@wlu.edu}}}
\address{Department of Mathematics, Washington and Lee University, Lexington, VA 24450, USA}

\author[S. J. Miller]{Steven J. Miller}
\email{\textcolor{blue}{\href{mailto:sjm1@williams.edu}{sjm1@williams.edu},
\href{mailto:Steven.Miller.MC.96@aya.yale.edu}{Steven.Miller.MC.96@aya.yale.edu}}}
\address{Department of Mathematics and Statistics, Williams College, Williamstown, MA 01267, USA}

\author[G. Tresch]{Garrett Tresch}
\email{\textcolor{blue}{\href{mailto:treschgd@tamu.edu}{treschgd@tamu.edu}}}
\address{Department of Mathematics, Texas A\&M University, College Station, TX 77840, USA}

\begin{document}

\thanks{We would like to thank the participants at the Polymath Jr.\ Program for helpful conversations.}

\subjclass[2020]{11D04 (primary); 11B39; 11B83  (secondary)}

\keywords{Diophantine equation; Fibonacci numbers; sequences}

\maketitle

\begin{abstract}
For two relatively prime positive integers $a, b\in \mathbb{N}$, it is known that exactly one of the two Diophantine equations 
$$ax + by \ =\ \frac{(a-1)(b-1)}{2}\ \mbox{ and }\ 1 + ax + by \ =\ \frac{(a-1)(b-1)}{2}$$
has a nonnegative integral solution $(x, y)$. Furthermore, the solution is unique. In this note, we summarize recent results and some new ones on the solution of the two equations and provide an overview of  problems for future investigation, some of which were presented at the 2025 International Conference on Class Groups of Number Fields and Related Topics.
\end{abstract}

\section{The two Diophantine equations with coefficients from special sequences}
In the study of cyclotomic polynomials, Beiter \cite{Be} showed that exactly one of the two Diophantine equations
$$ax + by \ =\ \frac{(a-1)(b-1)}{2}\ \mbox{ and }\ 1 + ax + by \ =\ \frac{(a-1)(b-1)}{2},$$
where $a$ and $b$ are prime numbers, has a nonnegative integral solution, and that solution is unique. Later, Chu \cite{C4} observed that the result still holds for relatively prime positive integers $a$ and $b$ and studied the solution when $a$ and $b$ are consecutive Fibonacci numbers. 

\begin{thm}\label{m1}\cite[Theorem 1.1]{C4}
Let $a,b\in\mathbb{N}$ be relatively prime.  Exactly one of the two equations
\begin{equation}\label{e11}ax + by \ =\ \frac{(a-1)(b-1)}{2}\ \mbox{ and }\ 1 + ax + by \ =\ \frac{(a-1)(b-1)}{2}\end{equation}
has a nonnegative integral solution. Furthermore, the solution is unique. 
\end{thm}

In this note, we summarize recent and new results concerning the solution of the two equations and outline directions for future research. The problems we propose are open-ended and accessible, yet rich enough to stimulate further investigation. We hope they will foster engagement among researchers who share a passion for Diophantine equations.

\subsection{Fibonacci numbers and their powers}
Define Fibonacci numbers using one of the standard conventions: $F_1 = F_2 = 1$ and $F_n = F_{n-1} + F_{n-2}$ for $n\ge 3$. Note that consecutive Fibonacci numbers are relatively prime because if $d = \gcd(F_{n-1}, F_n)$ for some $n\ge 2$, then the recurrence allows us to go backward to have $d = \gcd(F_1, F_2)$, which is $1$. This observation inspires \cite[Theorem 1.4]{C4} that gives the solution when $(a,b) = (F_n, F_{n+1})$; for example,   
\begin{align}\label{e10}
    &F_{6k}\frac{F_{6k-1}-1}{2} + F_{6k+1}\frac{F_{6k-1}-1}{2}\ =\ \frac{(F_{6k}-1)(F_{6k+1}-1)}{2},\mbox{ and}\nonumber\\
    &1 + F_{6k+4}\frac{F_{6k+4}-1}{2} + F_{6k+5}\frac{F_{6k+2}-1}{2}\ =\ \frac{(F_{6k+4}-1)(F_{6k+5}-1)}{2}, \mbox{ for all }k\ge 1.
\end{align}
More intriguing solutions appear when $F_n$ is replaced with $F_n^2$ or $F_n^3$. Below we sample two identities from  the work of a Polymath Jr.\ 2024's group; see \cite[Theorem 1.3 and 1.5]{CCKKLMYY} for all the cases: for $n\equiv 0, 2, 3, 5\mod 6$,
$$
    F_n^2\left(F_n^2 - \frac{F_{n-1}^2+1}{2}\right) + F^2_{n+1}\frac{F_{n-1}^2-1}{2}\ =\ \frac{(F_n^2-1)(F_{n+1}^2-1)}{2},
$$
and for $m\ge 2$, 
$$
F^3_{2m-1}\left(\sum_{k=1}^{2m-1}(-1)^{k-1}F_k^3\right) + F^3_{2m}\left(\sum_{k=2}^{2m-2}F_k^3\right)\ =\ \frac{(F^3_{2m-1}-1)(F^3_{2m}-1)}{2}.
$$
For $n\equiv 1, 4\mod 6$ and for $(a,b) = (F^3_{2m}, F^3_{2m+1})$, the second equation in \eqref{e11} is used. 
These findings motivate the question of whether there always exists a formula for the solution when $(a,b)$ are other powers of Fibonacci numbers. 

\begin{prob}
 For $(i,j)\in \mathbb{N}^2$, find the nonnegative integral solution $(x,y)$ to
 $$
     F_n^i x + F_{n+1}^{j} y \ =\ \frac{(F_n^i-1)(F^j_{n+1}-1)}{2}\mbox{ or }1 + F_n^i x + F_{n+1}^{j} y \ =\ \frac{(F_n^i-1)(F^j_{n+1}-1)}{2}.
     $$
Is it true that if $i = j = k$, then the solution must involve the $k$\textsuperscript{th} power of Fibonacci numbers?
\end{prob}

\subsection{Balancing numbers and Lucas-balancing numbers}
A positive integer $n$ is called \textit{balancing} (see \cite{BP}) if there is a nonnegative integer $d$ with
$$1+2+\cdots + (n-1)\ =\ (n+1) + \cdots + (n+d).$$
The sequence of balancing numbers $(B_n)_{n=1}^\infty$ satisfies $B_1 = 1$, $B_2 = 6$, and $B_n = 6B_{n-1}-B_{n-2}$ for $n\ge 3$. Meanwhile, the \textit{Lucas-balancing} numbers $(C_n)_{n=1}^\infty$ are recursively defined as $C_1 = 3$, $C_2 = 17$, and $C_n= 6C_{n-1} - C_{n-2}$ for $n\ge 3$. The motivation for Lucas-balancing numbers is that $(C_n)_{n=1}^\infty$ are related to $(B_n)_{n=1}^\infty$ in the same way that Lucas numbers are associated with Fibonacci numbers \cite{Pa}.  By finding the solution to \eqref{e11}, Davala \cite{Da} established intriguing identities that involved $B_n$'s and $C_n$'s; see \cite[Theorems 2.1--2.6]{Da}. Since $\gcd(B_{4n}, B_{4n+2}) = 6$, we cannot let $a = B_{4n}$ and $b = B_{4n+2}$ in \eqref{e11}; instead, \cite[Theorem 2.3]{Da} lets  $a = B_{4n}/6$ and $b = B_{4n+2}/6$. This motivates the following function $\Gamma: \mathbb{N}^2\rightarrow \{0,1\}$ introduced in \cite{ACLLMM}, where $\Gamma(a,b) = 0$ if 
$$\frac{a}{\gcd(a,b)}x + \frac{b}{\gcd(a,b)}y \ =\ \frac{(a/\gcd(a,b)-1)(b/\gcd(a,b)-1)}{2}$$
has the nonnegative integral solution, and $\Gamma(a,b) = 1$ if
$$1 + \frac{a}{\gcd(a,b)}x + \frac{b}{\gcd(a,b)}y \ =\ \frac{(a/\gcd(a,b)-1)(b/\gcd(a,b)-1)}{2}$$
has the nonnegative integral solution. 

\begin{prob}
   For $(a,b)\in \mathbb{N}^2$, investigate any relation between $\Gamma(a,b)$ and $\Gamma(a^2, b^2)$ or generally, $\Gamma(a^i, b^j)$ with $i,j\in \mathbb{N}$. 
\end{prob}

We discuss the function $\Gamma$ more in the next section.

\subsection{Fibonacci-like sequences}
A sequence $(t_n)_{n=1}^\infty$ is called \textit{Fibonacci-like} if the first two terms $t_1$ and $t_2$ are relatively prime positive integers, and $t_n = t_{n-1} + t_{n-2}$ for $n\ge 3$. It follows from the Fibonacci recurrence that $t_n = F_{n-2}t_1 + F_{n-1}t_2$ and $\gcd(t_n, t_{n+1}) = 1$ for all $n$. We are interested in the solution to 
\begin{equation}\label{e7}
    t_n x + t_{n+1}y \ =\ \frac{(t_n-1)(t_{n+1}-1)}{2}\ \mbox{ or }\ 1 + t_n x + t_{n+1}y \ =\ \frac{(t_n-1)(t_{n+1}-1)}{2}.
\end{equation}

Chu et al. \cite{CGGJMS} presented six cases corresponding to $n\mod 6$. For expository purpose, we focus on the case $n\equiv 4\mod 6$. 

\begin{thm}\cite[Theorem 3.4]{CGGJMS}
    Given $(u, v, n, r)\in \mathbb{Z}^4$ with even $n$, it holds that
        \begin{align}\label{e6}
        &\underbrace{\frac{1}{2}\left((u-r)F_{n-1}+\frac{(u-r)v-1}{u}F_n - 1\right)}_{=:\Phi^{(0)}(u, v, n, r)}t_n + \nonumber\\
        &\underbrace{\frac{1}{2}\left(rF_{n-2}+\frac{vr+1}{u}F_{n-1}-1\right)}_{=:\Psi^{(0)}(u,v,n,r)}t_{n+1}\ =\ \frac{(t_n-1)(t_{n+1}-1)}{2}.
    \end{align}
    and 
        \begin{align}\label{e5}
        1 + &\underbrace{\frac{1}{2}\left((u-r)F_{n-1}+\frac{(u-r)v+1}{u}F_n - 1\right)}_{=:\Phi^{(1)}(u, v, n, r)}t_n + \nonumber\\
        &\underbrace{\frac{1}{2}\left(rF_{n-2}+\frac{vr-1}{u}F_{n-1}-1\right)}_{=:\Psi^{(1)}(u,v,n,r)}t_{n+1}\ =\ \frac{(t_n-1)(t_{n+1}-1)}{2}
    \end{align}
\end{thm}

While Identities \eqref{e5} and \eqref{e6} are true for all integers $u,v,r$, and even $n$, the pairs $(\Phi^{(0)}(u, v, n, r),\Psi^{(0)}(u,v,n,r))$ and $(\Phi^{(1)}(u, v, n, r),\Psi^{(1)}(u,v,n,r))$ are not necessarily the desirable solution because they may not be nonnegative and integral. Since $u, v$, and $n$ are fixed, we need to choose a suitable $r$ that makes exactly one of the pairs the solution. The following lemma helps choose the right $r$.

\begin{lem}\label{l2} (cf.\ \cite[Lemmas 3.1 and 3.2]{CGGJMS})
    Let $(u,v)\in \mathbb{N}^2$ with $\gcd(u,v) = 1$. If $u$ is odd, then there exists a unique odd $r\in [1,u]$ with $vr\equiv \pm 1\mod u$. If $u$ is even, then there exists a unique odd $r\in [1,u]$ with $vr\equiv \pm 1\mod 2u$.
\end{lem}

Denote the $r$ in Lemma \ref{l2} by $\mathbb{O}(u,v)$. Then \cite[Theorem 3.7]{CGGJMS} states that the solution $(x,y)$ to \eqref{e7} when $n\equiv 4\mod 6$ is
$$(\Phi^{(1)}(u, v, n, \mathbb{O}(u,v)), \Psi^{(1)}(u,v,n, \mathbb{O}(u,v)))$$
if $$\begin{cases} u \mbox{ is odd and }v\mathbb{O}(u,v)\equiv 1\mod u,\mbox{ or }\\
u\mbox{ is even and }v\mathbb{O}(u,v)\equiv 1\mod 2u,\end{cases}$$
and $(x,y)$ is
$$(\Phi^{(0)}(u, v, n, \mathbb{O}(u,v)), \Psi^{(0)}(u,v,n, \mathbb{O}(u,v)))$$
if $$\begin{cases} u \mbox{ is odd}, u\ge 3,\mbox{ and }v\mathbb{O}(u,v)\equiv -1\mod u,\mbox{ or }\\
u\mbox{ is even and }v\mathbb{O}(u,v)\equiv -1\mod 2u.\end{cases}$$

For example, in the case of Fibonacci numbers, we have $u = v = 1$ and $\mathbb{O}(1,1) = 1$, so 
$$(x,y) \ =\ (\Phi^{(1)}(1, 1, n, 1), \Psi^{(1)}(1,1,n, 1))\ =\ \left(\frac{1}{2}(F_n-1),\frac{1}{2}(F_{n-2}-1)\right),$$
which recovers \eqref{e10}. 

\begin{prob}
Fix a recurrence relation other than the Fibonacci recurrence. 
    Find a formula that computes the solution $(x,y)$ to \eqref{e11} when $a$ and $b$ are taken from a sequence having the recurrence with different initial terms.
\end{prob}

\section{The function $\Gamma: \mathbb{N}^2 \rightarrow \{0,1\}$}

\subsection{The sequence $(\Gamma(a_n, a_{n+1}))_{n=1}^\infty$, given $(a_n)_{n=1}^\infty$}

Arachchi et al.\ \cite{ACLLMM} studied the sequence $(\Gamma(a_n, a_{n+1}))_{n=1}^\infty$ for special sequences $(a_n)_{n=1}^\infty$ including the $k$\textsuperscript{th}-power sequence $(n^k)_{n=1}^\infty$, arithmetic progressions, and shifted geometric sequences $(ar^{n-1} + 1)_{n=1}^\infty$. They showed that $(\Gamma(n^k, (n+1)^k))_{n=1}^\infty$ eventually alternates between $0$ and $1$, namely $0,1,0,1,0,1,\ldots$; if $(a_n)_{n=1}^\infty$ is an arithmetic progression, then $(\Gamma(a_n, a_{n+1}))_{n=1}^\infty$ alternates between $0$ and $1$ right away; last but not least, $(\Gamma(ar^{n-1}+1, ar^n+1))_{n=1}^\infty$ also alternates between $0$ and $1$ if $\gcd(a+1, r-1)$ is even. The proofs of these results used the following method to compute $\Gamma$. For two positive integers $a,b$ with $b/\gcd(a,b) > 1$, let $\Theta(a,b)$ be the unique multiplicative inverse of $a/\gcd(a,b)$ modulo $b/\gcd(a,b)$.

\begin{thm} \cite[Theorem 1.1]{ACLLMM}\label{mt}
 Let $a, b\in \mathbb{N}$. If $a|b$ or $b|a$, then $\Gamma(a,b) = 0$. Otherwise:
 \begin{enumerate}
     \item[a)] when $a/\gcd(a,b)$ is odd, then $\Gamma(a,b) = 0$ if and only if $\Theta(b,a)$ is odd;
     \item[b)] when $a/\gcd(a,b)$ is even, then $\Gamma(a,b) = 0$ if and only if $\Theta(a,b)$ is odd. 
 \end{enumerate}
\end{thm}

\begin{prob}
Define the collection of sequences
$$\mathcal{F} \ =\ \{(a_n)_{n=1}^\infty: (\Gamma(a_n, a_{n+1}))_{n=1}^\infty\mbox{ eventually alternates between }0 \mbox{ and }1\}.$$ Characterize $\mathcal{F}$ or large subsets of $\mathcal{F}$.   
\end{prob}

\begin{prob}
Fix a linear recurrence relation. Given a sequence $(a_n)_{n=1}^\infty$ that satisfies the linear recurrence relation, how do the initial conditions affect $(\Gamma(a_n,a_{n+1}))_{n=1}^\infty$?    
\end{prob}

For example, the sequence $(a_n)_{n=1}^\infty$ of all $1$'s satisfies the recurrence $a_n = 2a_{n-1}-a_{n-2}$ with initial conditions $a_1 = a_2 = 1$, and $(\Gamma(a_n, a_{n+1}))_{n=1}^\infty$ is constantly $0$. Meanwhile, the sequence of natural numbers satisfies the same recurrence with initial terms $1$ and $2$, but $(\Gamma(n,n+1))_{n=1}^\infty = 0,1,0,1,\ldots$ according to \cite[Theorem 1.6]{ACLLMM}.

\begin{prob}\label{2seqs}
 Given two sequences of positive integers $(a_n)_{n=1}^\infty$ and $(b_n)_{n=1}^\infty$, study the sequence $(\Gamma(a_n, b_n))_{n=1}^\infty$.   
\end{prob}

Observe that $(\Gamma(a_n, a_{n+1}))_{n=1}^\infty$ is a special case of Problem \ref{2seqs} when $b_n = a_{n+1}$ for every $n\in \mathbb{N}$.

\subsection{Given split}
Returning to the original result, we know that if $a$ and $b$ are relatively prime positive integers, then exactly one of the two equations $$i + ax + by \ = \ \frac{(a-1)(b-1)}{2}, \ \ \ i \in \{0, 1\}$$ has a positive integral solution. For a given $a$, the following result tells us how to choose $b$ to get the desired $\Gamma(a,b)$.

\begin{prop}\cite[Corollary 2.1]{ACLLMM}\label{pp1}
For $a\ge 1$, $\Gamma(a, 2a) = 0$; for $a\ge 2$, $\Gamma(a, 2a-1) = 1$.
\end{prop}

Our new result constructs, for every $p\in [0,1]$, a strictly increasing sequence $(a_n)_{n=0}^\infty$ so that as $N \to \infty$, for $n \le N$, the percentage of the time that $\Gamma(a_{n-1}, a_n) = 0$ converges to $p$. 

\begin{cor}\label{cc1}
For every $p\in [0,1]$, there exists a strictly increasing sequence $(a_n)_{n=0}^\infty$ so that 
$$\lim_{N\rightarrow\infty}\frac{\#\{n\le N\,:\, \Gamma(a_{n-1},a_n) = 0\}}{N} \ =\ p.$$
\end{cor}

\begin{proof}
Let $p \in [0,1]$. Thanks to Proposition \ref{pp1}, if $p = 1$, we choose $(a_n)_{n=0}^\infty$ to be $(2^n)_{n=0}^\infty$. If $p = 0$, we choose $(a_n)_{n=0}^\infty$ to be $(2^n+1)_{n=0}^\infty$ because then for every $n\ge 1$,
$$a_{n}  \ =\ 2^n+1 \ =\ 2(2^{n-1}+1)-1 \ =\ 2a_{n-1}-1.$$

Assume that $p\in (0,1)$. Let $a_0 = 1$ and $a_1 = 2$. For $n\ge 2$, we construct $a_n$ inductively as follows: 
$$a_n\ =\ \begin{cases} 2a_{n-1}, &\mbox{ if } b_{n-1}:=\frac{\#\{j\le n-1\,:\, \Gamma(a_{j-1},a_j) = 0\}}{n-1} < p;\\  2a_{n-1}-1, &\mbox{ otherwise}.\end{cases}$$
Here we follow the greedy algorithm where the next term $a_n$ is chosen based on whether the ratio of interest $b_{n-1}$ has exceeded $p$. If $b_{n-1}\ge p$, then the pair $(a_{n-1}, a_{n})$ satisfies $\Gamma(a_{n-1}, a_n) = 1$ so that $b_n < b_{n-1}$; otherwise, $\Gamma(a_{n-1}, a_n) = 0$, making $b_n > b_{n-1}$.

Let us take a closer look at the sequence
$$\left(b_N\ :=\ \frac{\#\{n\le N\,:\, \Gamma(a_{n-1},a_n) = 0\}}{N}\right)_{N=1}^\infty.$$
By construction, there are infinitely many $N$'s with $b_N < p$; similarly, there are infinitely many $N$'s with $b_N \ge p$. Let $(N_j)_{j=1}^\infty$ be the subsequence of $\mathbb{N}_{\ge 2}$ such that for each $N\in (N_j)_{j=1}^\infty$, either
\begin{itemize}
    \item $b_{N-1} < p$ and $b_{N}\ge p$, or
    \item $b_{N-1} \ge p$ and $b_{N} < p$. 
\end{itemize}
In other words, $N_j$ is the $j^{\sf{th}}$ index for when the ratio, $(b_N)_{N=1}^\infty$, changes from being below $p$ to above or at $p$ or vice versa.

To show that $b_N\rightarrow p$, pick $\varepsilon > 0$. Choose $N_m$ with $1/N_m < \varepsilon$. We claim that for all $N\ge N_m$, $|b_N - p| < \varepsilon$. We proceed by case analysis.\ \\

Case 1: $N = N_j$ for some $j\ge m$ with $b_{N_j}\ge p$. We have
\begin{align*}
0\ \le\ b_{N_j} - p &\ =\ \frac{\#\{n\le N_j\,:\, \Gamma(a_{n-1},a_n) = 0\}-pN_j}{N_j}\\
&\ \le\ \frac{\#\{n\le N_j-1\,:\, \Gamma(a_{n-1},a_n) = 0\} + 1-p(N_j-1)}{N_j}\\
&\ =\ \frac{(b_{N_j-1}-p)(N_j-1)+1}{N_j}\ <\ \frac{1}{N_j}\ <\ \varepsilon.
\end{align*}
\\ \

Case 2: $N = N_j$ for some $j\ge m$ with $b_{N_j} < p$. We have
\begin{align*}
0 \ <\ p - b_{N_j} &\ =\ \frac{pN_j- \#\{n\le N_j\,:\, \Gamma(a_{n-1},a_n) = 0\}}{N_j}\\
&\ \le\ \frac{p(N_j-1) - \#\{n\le N_j-1\,:\, \Gamma(a_{n-1},a_n) = 0\} + p}{N_j}\\
&\ =\ \frac{(p-b_{N_j-1})(N_j-1)+p}{N_j}\ <\ \frac{1}{N_j}\ <\ \varepsilon.
\end{align*}
\\ \

Case 3: $N_j < N < N_{j+1}$ for some $j\ge m$. 
\begin{enumerate}
    \item Case 3.1: If $b_{N_j}\ge p$, then $b_N \ge p$ and  $b_{N_j} > b_N$. It follows from Case 1 that
    $$0\ \le\ b_N - p\ <\ b_{N_j}-p\ <\ \varepsilon.$$
    \item Case 3.2: If $b_{N_j} < p$, then $b_N < p$ and $b_N > b_{N_j}$. It follows from Case 2 that
    $$0\ <\ p-b_N \ <\ p-b_{N_j}\ <\ \varepsilon.$$
\end{enumerate}

\mbox{ }\\ \
We have shown that for all $N\ge N_m$, $|b_N - p| < \varepsilon$ and thus, $b_N\rightarrow p$.
\end{proof}

\begin{rek}\label{rr1}
For every $p\in [0,1]$, the sequence $(a_n)_{n=0}^\infty$ constructed in the proof of Corollary \ref{cc1} satisfies 
$$a_n\ \le\ 2^{n+1}, \mbox{ for all }n\ge 0.$$
On the other hand, since either $a_n = 2a_{n-1}$ or $a_n = 2a_{n-1}-1$, we have
$$a_n\ \ge\ (a_1-1)2^{n-1} + 1\ >\ 2^{n-1}, \mbox{ for all }n\ge 1.$$
This follows from simple algebra:
\begin{align*}
a_n & \ \ge \  2 a_{n-1} - 1 \nonumber\\
&\ \ge \ 2^2 a_{n-2} - 1 - 2 \nonumber\\
& \ \ge \ 2^3 a_{n-3} - 1 - 2 - 2^2 \nonumber\\
& \mbox{ }\mbox{ }\mbox{ } \vdots\  \nonumber\\
& \ \ge\   2^{n-1} a_1 - 1 - 2 - \cdots - 2^{n-2} \nonumber\\
&\ = \ 2^{n-1} a_1 - 2^{n-1} + 1 \ = \ 2^{n-1} (a_1 - 1) + 1 \ > \ 2^{n-1}. \nonumber
\end{align*}
Hence, $(a_n)_{n=0}^\infty$ grows in the order of $2^n$.
\end{rek}

\subsection{Periodicity of $(\Gamma(k, a_n))_{n=1}^\infty$}

For a fixed $k\in \mathbb{N}$ and a sequence of positive integers $(a_n)_{n=1}^\infty$, let $T_k((a_n)_{n=1}^\infty)$ denote the eventual period (if it exists) of the sequence $(\Gamma(k, a_n))_{n=1}^\infty$. Trivially, if $k = 1$, we have a constant sequence of $0$'s with period $1$. Arachchi et al.\ \cite{ACLLMM} used Theorem \ref{mt} to compute $T_{k}(\mathbb{N})$.

\begin{thm}\cite[Theorem 1.8]{ACLLMM}\label{ThmNats} Fix $k\in \mathbb{N}$. The following hold.
\begin{enumerate}
    \item If $k$ is odd, $T_k(\mathbb{N}) = k$; furthermore, in each period, the number of $0$'s is one more than the number of $1$'s.
    \item If $k$ is even, $T_k(\mathbb{N}) = 2k$; furthermore, in each period, the number of $0$'s is two more than the number of $1$'s.
\end{enumerate}
\end{thm}

We emphasize the key observations within the proof of Theorem \ref{ThmNats}.

\begin{rek}\label{ResNats}
Within Theorem \ref{ThmNats}, the following facts are shown:
\begin{enumerate}
    \item if $k$ is odd, then $\Gamma(k, s) \neq \Gamma(k, k - s)$ for all $1 \leq s \leq (k - 1)/2$;
    \item if $k$ is even, then $\Gamma(k, s) \neq \Gamma(k, 2k - s)$ for all $1 \leq s \leq k - 1$.
\end{enumerate}
Hence, once we know the first half of each period, we know the second half. 
\end{rek}

Let us consider $\Gamma(k, (a_n)_{n=1}^\infty)$ for some certain $k$ and arithmetic progressions $(a_n)_{n=1}^\infty$.

\begin{prop}\label{Prop:Arithwithrelprime}
    Define the arithmetic sequence $(a_n:=pn-r)_{n=1}^\infty$ for some $p\in \mathbb{N}$ and $r\in \mathbb{N}\cup\{0\}$ with $0\leq r<p$. Let $k\in \mathbb{N}$ be odd. If $\gcd(k,p)=1$, then $T_k((a_n)_{n=1}^\infty)=k$, and within each period, the number of 0's is one more than the number of 1's.
\end{prop}

\begin{proof}
Since $\gcd(k,p)=1$, for each $n\in \{1,2,\dots, k\}$, there is a unique integer $x_n\in [1,k]$ with $a_{x_n}\equiv n\ \mod k$. Let $(\ell_n)_{n=1}^k$ be nonnegative integers such that
$$a_{x_n}\ =\ \ell_n k + n, \quad 1\le n\le k.$$
By Theorem \ref{ThmNats}, for  $1\leq n\leq (k-1)/2$,
$$
    \Gamma(k,a_{x_n})\ =\ \Gamma(k,\ell_nk+n)\ =\ \Gamma(k,n)
$$
and
$$
\Gamma(k,a_{x_{k-n}})\ =\ \Gamma(k, \ell_{k-n}k + (k-n))\ =\ \Gamma(k, k-n).
$$
According to Remark \ref{ResNats} item (1), $\Gamma(k,n)\neq \Gamma(k, k-n)$, so
$$\Gamma(k, a_{x_n})\ \neq\ \Gamma(k, a_{x_{k-n}}).$$
Together with the fact that 
$$\Gamma(k,a_{x_k})\ =\ \Gamma(k, \ell_k k+k)\ =\ 0,$$
we conclude that within the first $k$ terms of $(\Gamma(k, a_n))_{n= 1}^\infty$, the number of 0’s is one more than the number of 1’s.

It remains to show that $T:=T_k((a_n)_{n=1}^\infty) = k$. By Theorem \ref{ThmNats}, we have
$$\Gamma(k, a_{n+k})\ =\ \Gamma(k, p(n+k)-r)\ =\ \Gamma(k, a_n + kp)\ = \ \Gamma(k, a_n),$$
so $T$ divides $k$. Hence, within the first $k$ terms, there are $k/T$ copies of the period. Let $a$ and $b$ be the number of 0’s and 1’s within each period, respectively. Then $(a - b)(k/T) = 1$, which implies that $a - b = k/T = 1$. Therefore, $T = k$, as desired.
\end{proof}

\begin{prop}{\label{ThmOdds}}
For every $k \in \mathbb{N}$, $T_{2^k}((2n-1)_{n=1}^\infty)=2^k$. Furthermore, in each period, there are $2^{k-1}$ many 0's and $2^{k-1}$ many 1's.
\end{prop}

\begin{proof}
Fix $k\in \mathbb{N}$. We have 
$$2^k \cdot 0 + (2^k-1)\cdot (2^{k-1}-1)\ =\ \frac{(2^k-1)(2^k-2)}{2},$$
so 
\begin{equation}\label{er1}\Gamma(2^k, 2^{k}-1)\ =\ 0\ =\ \Gamma(2^k, 1).\end{equation} Let $T:=T_{2^k}((2n-1)_{n= 1}^\infty)$. It is easy to see that 
$T$ divides $2^k$. Indeed, for every $n\in \mathbb{N}$, Theorem \ref{ThmNats} gives
$$\Gamma(2^k, 2(n+2^k)-1)\ =\ \Gamma(2^k, (2n-1)+2^{k+1})\ =\ \Gamma(2^k, 2n-1).$$
Suppose, for contradiction, that $2^k=\ell T$ for some even $\ell\ge 2$. By \eqref{er1} and Remark \ref{ResNats} item (2),
\begin{align*}\Gamma(2^k,1)\ =\ \Gamma(2^k,2^k-1)\ \neq\ \Gamma(2^k,2^{k+1}-(2^k-1))&\ =\ \Gamma(2^k,2^k+1)\\
&\ =\ \Gamma\left(2^k,2\left(1+\frac{\ell}{2}T\right)-1\right).
\end{align*}
Hence, 
$$\Gamma(2^k, 2\cdot 1 - 1)\ \neq\ \Gamma\left(2^k,2\left(1+\frac{\ell}{2}T\right)-1\right),$$
which contradicts that $T$ is the period of $\Gamma(2^k, (2n-1)_{n=1}^\infty)$.

The second statement follows immediately from Remark \ref{ResNats} item (2).
\end{proof}

Propositions \ref{Prop:Arithwithrelprime} and \ref{ThmOdds} motivate the following problem. 

\begin{prob}\label{Prob:Arithperiod}
For a general arithmetic progression $a_n=pn-r$ with $0\leq r<p$ and $k\geq 1$, compute $T_k((a_n)_{n=1}^\infty)$.
\end{prob}

For more complex sequences, another direction is to study the possible eventual periodic behavior of $(\Gamma(k,a_n))_{n=1}^\infty$.

\begin{prob}\label{perpr}
 For what sequences $(a_n)_{n=1}^\infty$, is $(\Gamma(k, a_n))_{n=1}^\infty$ eventually periodic for all $k\in\mathbb{N}$? For such sequences $(a_n)_{n=1}^\infty$, how long does it take for $(\Gamma(k, a_n))_{n=1}^\infty$ to be periodic?
\end{prob}

\begin{thm}\label{nt}
If $(a_n)_{n=1}^\infty$ is a sequence of positive integers, and there exist $(c_1, \ldots, c_s)\in \mathbb{Z}^s$ and $(t_1,\ldots, t_s)\in \mathbb{Z}^s_{\ge 0}$ such that for sufficiently large $n$, $(a_n)_{n=1}^\infty$ is increasing, and
\begin{equation}\label{e13}a_n \ =\ c_1a^{t_1}_{n-1} + c_2a^{t_2}_{n-2} + \cdots + c_sa^{t_s}_{n-s},\end{equation}
then $(\Gamma(k, a_n))_{n=1}^\infty$ is eventually periodic for every $k\in \mathbb{N}$. In particular, if we define $\pi(k)$ to be the (eventual) period of $(a_n \mod k)_{n=1}^\infty$, then $T_k((a_n)_{n=1}^\infty)$ divides $\pi(2k)$.
\end{thm}

While Theorem \ref{nt} gives a large class of sequences $(a_n)_{n=1}^\infty$ with eventually periodic $(\Gamma(k, a_n))_{n=1}^\infty$, there are other sequences that answer Problem \ref{perpr}. For example, the sequence $a_n= (n!)^{n!}$ does not satisfy any relation like \eqref{e13} but $(\Gamma(k, a_n))_{n=1}^\infty$ is eventually $0$. To see this, it suffices to verify that for each fixed $(s, t)\in \mathbb{N}^2$,
$$\lim_{n\rightarrow\infty} \frac{(n+s)!^{(n+s)!}}{(n!)^{t\cdot n!}}\ =\ \infty.$$
Indeed, for $n\ge 2$, we have
$$\frac{(n+s)!^{(n+s)!}}{(n!)^{t\cdot n!}} \ >\ (n!)^{(n+s)! - t\cdot n!}\ \ge\ 2^{n!(n+1 - t)}\rightarrow\infty.$$

To prove Theorem \ref{nt}, we need the following simple lemma.

\begin{lem}\label{l1}
    Let $a,b\in \mathbb{N}$ with $\gcd(a,b) = 1$. Let $x^*$ and $y^*$ be the nonnegative integers such that 
    $\delta + ax^*+by^* = (a-1)(b-1)/2$, where $\delta = \Gamma(a,b)$. If $N$ is a nonnegative integer such that $\gcd(a, b+N) = 1$ and $N((a-1)/2-y^*)$ is an integer divisible by $a$, then $\Gamma(a, b+N) = \Gamma(a,b)$.
\end{lem}

\begin{proof}
    Since $\delta + ax^* + by^* = (a-1)(b-1)/2$, we have
    $(a-1)(b-1)/2\ge by^*$. If $b = 1$, then $y^* = 0$. If $b > 1$, then $(a-1)/2 \ge by^*/(b-1)\ge y^*$. In both cases, $(a-1)/2-y^*\ge 0$. Hence,
    $$T\ :=\ \left(\frac{(a-1)(b-1)}{2}-by^*-\delta\right) + N\left(\frac{a-1}{2}-y^*\right)\ \ge\ 0.$$
    The hypotheses also guarantee that $T$ is divisible by $a$. Therefore, there exists a nonnegative integer $x'$ with
    $$ax' \ =\ \left(\frac{(a-1)(b-1)}{2}-by^*-\delta\right) + N\left(\frac{a-1}{2}-y^*\right),$$
    which is equivalent to
    $$\delta + ax' + (b+N)y^*\ =\ \frac{(a-1)(b+N-1)}{2}.$$
\end{proof}

\begin{proof}[Proof of Theorem \ref{nt}]
    Let $p\in \mathbb{N}$ be such that $(a_n)_{n=p}^\infty$ is increasing, and \eqref{e13} holds for $n\ge p$. Let $k\in \mathbb{N}$ and 
    $$M_s \ :=\ \{(a_{n}, a_{n+1}, \ldots, a_{n+s-1})\,:\, n\ge p-s\}.$$
    Modulo $k$, there are $k^s$ configurations of elements in $M_s$. Since the set $M_s$ is infinite, there exist $n_1$ and $n_2$ with $p - s\le n_1 < n_2$ and 
    \begin{equation}\label{e12}(a_{n_1}, a_{n_1+1}, \ldots, a_{n_1+s-1}) \ \equiv\ (a_{n_2}, a_{n_2+1}, \ldots, a_{n_2+s-1})\mod k.\end{equation}
    We prove by induction that for every $n\ge n_1$,
    $$a_n\ \equiv\ a_{n+n_2-n_1}\mod k.$$
    By \eqref{e12}, $a_{n}\equiv a_{n+n_2-n_1}\mod k$ for $n_1\le n\le n_1+s-1$. Assume that there exists $j\ge n_1+s-1$ such that  $a_n \equiv a_{n+n_2-n_1}\mod k$ for all $n\in [n_1, j]$. The Recurrence \eqref{e13} guarantees that $a_{j+1}\equiv a_{j+1+n_2-n_1}\mod k$. We have shown that for every $k\in \mathbb{N}$, the sequence $(a_n\mod k)_{n=1}^\infty$ is eventually periodic. Let the period be denoted by $\pi(k)$.

    For some $m\in \mathbb{N}$, let $(a_n\mod 2k)_{n=m}^\infty$ be periodic with period $\pi(2k)$. Since for all $n\ge \max\{p, m\}$, $a_{n+\pi(2k)}-a_n$ is divisible by $2k$, we know that $\gcd(k, a_{n}) = \gcd(k, a_{n+\pi(2k)}) =: d$. Applying Lemma \ref{l1} with $a = k/d$, $b = a_n/d$, and $N = (a_{n+\pi(2k)}-a_n)/d$. If we write $a_{n+\pi(2k)}-a_n = 2kj$, then 
    $$N\left(\frac{k/d-1}{2}-y^*\right)\ =\ \frac{jk}{d}\left(\frac{k}{d}-1-2y^*\right),$$
    which is divisible by $a$. By Lemma \ref{l1}, 
    $\Gamma(k, a_n) = \Gamma(k, a_{n+\pi(2k)})$. Hence, $T_k((a_n)_{n=1}^\infty)$ divides $\pi(2k)$. 
\end{proof}

As an application of Theorem \ref{nt}, we know that $T_k((F_n)_{n=1}^\infty)$ is a divisor of the Pisano period of the Fibonacci sequence. For $k\in \mathbb{N}$, the $k$\textsuperscript{th} Pisano period, denoted by $\pi(k)$, is the period with which the Fibonacci sequence modulo $k$ repeats. For example, it is known that $\pi(2^j) = 3\cdot 2^{j-1}$ for every $j\in \mathbb{N}$, $\pi(3) = 8$, and $\pi(5) = 20$. Table \ref{Bn} compares the values of $T_k((F_n)_{n=1}^\infty)$ with $\pi(2k)$. 

\begin{table}[H]
\centering
\begin{tabular}{ |c| c| c||c|c|c|}
\hline
$k$ &$T_k((F_n)_{n=1}^\infty)$& $\pi(2k)$ & $k$ &$T_k((F_n)_{n=1}^\infty)$& $\pi(2k)$ \\
\hline
$1$ & $1$ & $3$ & $6$ & $24$ & $24$ \\
\hline
$2$ & $6$ & $6$ & $7$ & $16$ & $48$ \\
\hline
$3$ & $8$ & $24$ & $8$ & $24$ & $24$ \\
\hline
$4$ & $12$ & $12$ & $9$ & $24$ & $24$ \\
\hline
$5$ & $20$ & $60$ & $10$ & $60$ & $60$ \\
\hline
\end{tabular}
\caption{The first $10$ values of $T_k((F_n)_{n=1}^\infty)$ and $\pi(2k)$ with $1\le k\le 10$.}
\label{Bn}
\end{table}

\begin{cor}
    For each $k\in \mathbb{N}$, $T_{k}((F_n)_{n=1}^\infty)$ divides $\pi(2k)$.
\end{cor}

\begin{prob}
For each $k$, compute $T_{k}((F_n)_{n=1}^\infty)$.      
\end{prob}

\section{Future Work}

Besides the problems already mentioned, we end with a few naturally related questions to investigate; we hope to explore several of them with additional summer research students over the next few years, and invite readers who are interested in such problems to contact us.

\subsection{More recurrences}

Can we axiomatize the key steps in our proofs and classify other recurrence relations which can be analyzed by a similar argument? In particular, what properties of the investigated sequences are essential, from the defining recurrence to the initial conditions.

\subsection{More variables}

The most natural generalization is to increase the number of variables. The original result states that for every pair of relatively prime integers $a$ and $b$,  exactly one of the two equations $$i + ax + by \ = \ \frac{(a-1)(b-1)}{2}, \ \ \ i \in \{0, 1\}$$ has a solution in non-negative integers, and the solution is unique. 

We can extend the system to $n$ variables and $n$ equations 
\begin{equation}\label{ee1}i + a_1 x_1 + \cdots + a_n x_n \ = \ \frac{(a_1 - 1) \cdots (a_n - 1)}{2}, \ \ \ i \in \{0, 1, \dots, n-1\}.\end{equation} For a fixed $n$, for what $n$-tuples $(a_1, \dots, a_n)$ of positive, relatively prime integers does only one of the $n$ equations have a solution $(x_1, \dots, x_n)$ in non-negative integers? What is the relation between the tuple and the equation that has a solution? If $n=2$, we know that all such $n$-tuples work.

Note that in \eqref{ee1}, we kept the denominator as $2$.
As the $a$'s are relatively prime as long as we have at least two  of them then at least one of them must be odd, and thus these fractions are integers.

\subsection{Varying parameters}

We may modify the equations by choosing two integers $r, s$ and replacing $(a-1)(b-1)$ on the right hand side with $(a-r)(b-s)$. We can now ask many questions. 

Can we find pairs $(r, s)$ such that exactly one equation has a nonnegative solution, and the solution is unique? For such a pair $(r,s)$, find an analog of Theorem \ref{mt} and determine how often the first equation has a solution. Note some care may be required to ensure that the fractions are integers (if it is not an integer then there cannot be a solution).

As the above may be too strong of a condition in general, we consider a natural weakening. Instead of insisting that for every pair of relatively prime positive integers, exactly one of the two has a solution in non-negative integers, we require that to hold for at least $p$\% of the time for some $p \in [0, 1]$.

In honor of \cite{Be}, we call a pair $(r, s)$ a $p$-Beiter-good pair if for at least $p$\% of relatively prime integers $(a, b)$ with $0 < a, b \le x$ (as $x$ goes to infinity), exactly one of the two equations has a solution. Classify such pairs; clearly, if a pair is $p$-Beiter-good, it is also $q$-Beiter good for every $q < p$. 

Alternatively, can we find pairs $(r, s)$ such that both equations always have solutions? Such that neither has a solution?  

\subsection{Slower growth rate for given splits}
For every $p\in [0,1]$, Corollary \eqref{cc1} constructs an increasing sequence $(a_n)_{n=0}^\infty$ recursively with 
$$\lim_{N\rightarrow\infty}\frac{\#\{n\le N\,:\, \Gamma(a_{n-1},a_n) = 0\}}{N} \ =\ p.$$
As noted in Remark \ref{rr1}, the sequence grows exponentially fast (faster than
the Fibonacci sequence). Is there a construction involving a strictly increasing sequence that grows slower than the Fibonacci sequence, or perhaps even like $n^k$ for some $k$? 
What are some interesting values of $p$ that give natural sequences $(a_n)_{n=1}^\infty$?


\ \\ 

\end{document}